\documentclass[a4paper,xcolor]{article}
\usepackage{url,lineno}
\usepackage{amsfonts}
\usepackage{amssymb,amsthm,amsmath,graphicx,bm,ebezier,mathtools}
\usepackage{hyperref}
\usepackage{color,enumerate}
\usepackage{url,lineno}
\usepackage{arydshln}
\usepackage{lscape}
\usepackage{listings}
\usepackage[ruled,vlined]{algorithm2e}
\usepackage{soul} 
\usepackage{tikz}

\newtheorem{theorem}{Theorem}

\newtheorem{proposition}[theorem]{Proposition}
\newtheorem{corollary}[theorem]{Corollary}
\newtheorem{lemma}[theorem]{Lemma}

\theoremstyle{remark}

\newtheorem{example}[theorem]{Example}

\def\mod{{\rm\,mod\,}}

\def\N{\mathbb{N}}
\def\Z{\mathbb{Z}}

\def\k{\mathbf{k}}
\def\a{\mathbf{a}}
\def\b{\mathbf{b}}
\def\bS{\overline{S}}
\def\tS{\widetilde{S}}

\newcommand{\FSN}{\mathrm{FS}(\N)}

\title{On the ideals of some sumset semigroups}

\author{
J.\ I.\ Garc\'{\i}a-Garc\'{\i}a\footnote{
Departamento de Matem\'aticas/INDESS (Instituto Universitario para el Desarrollo Social Sostenible),
Universidad de C\'adiz, E-11510 Puerto Real (C\'{a}diz, Spain).
E-mail: ignacio.garcia@uca.es.}\\
D.\ Mar\'{\i}n-Arag\'on\footnote{
Departamento de Matem\'aticas,
Universidad de C\'adiz, E-11510 Puerto Real (C\'{a}diz, Spain).
E-mail: daniel.marin@uca.es.}
\\
A.\ Vigneron-Tenorio\footnote{
Departamento de Matem\'aticas/INDESS (Instituto Universitario para el Desarrollo Social Sostenible), Universidad de C\'adiz,
E-11406 Jerez de la Frontera (C\'{a}diz, Spain).
E-mail: alberto.vigneron@uca.es.}
}

\date{}

\begin{document}

\maketitle

\begin{abstract}
A sumset semigroup is a non-cancellative commutative monoid obtained from the sumset of finite non-negative integer sets. In this work, an algorithm for computing the ideals associated with some sumset semigroups is provided. 
Using these ideals, we study some factorization properties of sumset semigroups and some additive properties of sumsets. This approach links computational commutative algebra with additive number theory.
\end{abstract}

{\small


}

\section*{Introduction}

Additive number theory is the subfield of number theory concerning the study of subsets of integers and their behaviour under addition. More abstractly, the field of additive number theory includes the study of abelian groups and commutative semigroups with an operation of addition. The principal objects of study are (i)~the sumset of two subsets $A$ and $B$ of elements from an abelian group $G$, $A+B=\{a+b\mid a\in A,~b\in B\}$, and (ii)~to determine the structure and properties of the $h$-fold sumset $hA$ when the set $A$ is known. In an inverse problem, we start with the sumset $hA$ and try to deduce information about the underlying set $A$. An up-to-date reference for inverse problems can be found in \cite[Chapter~5]{tao_additive_2006}. There is a beautiful and straightforward solution of the direct problem of describing the structure of the $h$-fold sumset $hA$ for any finite set $A$ of integers and for all sufficiently large $h$ (see \cite[Theorem~1.1]{Nathanson}). This result has implications for the study of Weierstrass semigroups, such as is shown in \cite{buchweitz}.

Here, we consider the commutative semigroup whose elements are the finite subsets of $\N$, denoted by $(\FSN,+)$ with $+$ the operation defined as before. This semigroup is the power monoid of $\N$ (see \cite{FGKT}, \cite{Fan_Tringali} and the references therein). A sumset semigroup is a semigroup generated by a finite number of elements of $\FSN$. We show that the sumset semigroups are atomic reduced semigroups with finite elasticity. It is well known that finitely generated semigroups are finitely presented (see \cite{gilmer}). That is, there exists $p\in\N$ and a congruence $\sigma$ in $\N^p\times \N^p$ such that the semigroup $S$ is isomorphic to $\N^p/\sigma$. Equivalently, there exists a binomial ideal $I_S$ in the polynomial ring $\k[x_1,\dots,x_p]$ such that $S$ is isomorphic to the set of monomials in $\k[x_1,\dots,x_p]/I_S$ with the product operation. The presentation of $S$ (or a system of generators of the semigroup ideal $I_S$) provides us with a way to obtain the expressions of an element of the semigroup in terms of its generators. This can be done with Gr\"obner bases and related techniques. For instance, we can check whether the $h$-fold of any $A\in S$ can be expressed as a sumset of other elements. By using these techniques, we can build a bridge between computational commutative algebra and additive number theory.

We also present a new Python library \cite{commutative} that includes an implementation of our algorithms and the examples that illustrate it.

In this work, we show some properties of the semigroup $\FSN$ and give the ideals of some types of sumset semigroups. The work is organized as follows. In Section~\ref{preliminaries}, we present some definitions and results on Gr\"obner bases. In Section~\ref{sumset semigroups}, we introduce 
the sumset semigroups and study some of their properties. In Sections~\ref{ideals} and \ref{computing}, by using algebraic commutative algebra tools, we study the ideals of some families of sumset semigroups, thereby allowing us to introduce Algorithm~\ref{algoritmo_ideal_familia} and provide some examples.

\section{Some results on commutative algebra}\label{preliminaries}

For a field $\k$ and a set of indeterminates $\{x_1,\ldots ,x_t\}$, the polynomial ring $\k[x_1,\ldots ,x_t]$ (also denoted by $\k[X]$) is the set of polynomials in $\{x_1,\ldots ,x_t\}$ with coefficients in $\k$, that is, the set $ \{\sum_{i=1}^m a_i x_1^{\alpha_1}\cdots x_t^{\alpha_t} \mid m\in \N,\, a_i\in\k,\, \alpha_1,\ldots ,\alpha_t\in \N \}$. We denote by $X^\alpha$ the monomial $x_1^{\alpha_1}\cdots x_t^{\alpha_t}$, with $\alpha=(\alpha_1,\ldots ,\alpha_t)\in \N^t$. In this work, some results use Gr\"{o}bner basis theory and the Elimination Theorem. The necessary background can be found in \cite[$\S$2 and $\S$3]{CoxLittleOShea} but is also provided here so that the present work is self-contained.

It is well known that any ideal in a polynomial ring is finitely generated. In particular, there exists a special generating set associated with the ideals, namely a Gr\"{o}bner basis. This concept depends on the election of an order on the monomials. A monomial order $\prec$ on $\k [X]$ is a multiplicative total order on the set of monomials if for each two monomials $X^\alpha,X^\beta$ such that $X^\alpha\prec X^\beta$, then $X^\alpha X^\gamma \prec X^\beta X^\gamma$ for every monomial $X^\gamma$.

For a fixed monomial order $\prec$ on $\k[X]$, $\mathrm{In}_\prec(I)$ denotes the set of leading terms of non-zero elements of $I$, and $\langle \mathrm{In}_\prec(I)\rangle$ the monomial ideal generated by $\mathrm{In}_\prec(I)$. A subset $G$ of $I$ is a Gr\"{o}bner basis of $I$ if $\langle \mathrm{In}_\prec(I)\rangle = \langle \{ \mathrm{In}_\prec(g) \mid g\in G\}\rangle$, where $\mathrm{In}_\prec(g)$ is the leading term of $g$. An algorithm for computing a Gr\"{o}bner basis for $I$ is given in \cite[Chapter~2, $\S$7]{CoxLittleOShea}. It is also well known that Gr\"{o}bner bases of binomial ideals are sets of binomials.

Given two polynomials $f$ and $g$, their $S$-polynomial is defined by $S(f,g)= \frac{a X^\alpha}{\mathrm{In}_\prec(f)}f - \frac{a X^\alpha}{\mathrm{In}_\prec(g)} g$, where $aX^\alpha$ is the least common multiple of $\mathrm{In}_\prec(f)$ and $\mathrm{In}_\prec(g)$. Also, for $k$ non-zero polynomials $f_1, \ldots ,f_k$, we say that $S(f_i,f_j)=\sum _{l=1}^k g_lf_l$ has an {\em lcm} representation if the least common multiple of the monomial leaders of $f_i$ and $f_j$ is bigger than $\mathrm{In}_\prec(g_lf_l)$ (respect $\prec$) whenever $g_lf_l\neq 0$. Thus, we obtain another equivalent definition of a Gr\"{o}bner basis.

\begin{theorem}\cite[Chapter~2, $\S$9, Theorem~6]{CoxLittleOShea}
A basis $\{f_1,\ldots f_k \}$ of an ideal $I$ is a Gr\"{o}bner basis if and only if for every $i\neq j$, the $S$-polynomial $S(f_i,f_j)$ has an {\em lcm} representation.
\end{theorem}

The above theorem allows us to prove the next result. We use this lemma in the following sections.

\begin{lemma}\label{para_teorema}
Let $I\subset \k[X, Y]$ and $J\subset \k[Z,X,Y]$ be two binomial ideals with $J$ generated by $G=\{z_1-M_1, \ldots ,z_t-M_t\}$, where each $M_i$ is a monomial in $\k[X,Y]$. Fix $\prec$ a monomial order on $\k[Z,X,Y]$ such that $x_j,y_k\prec z_i$, for every $x_j,y_k$ and $z_i$. Then, the union of $G$ and a Gr\"{o}bner basis of $I$ respect $\prec$ is a Gr\"{o}bner basis of $I+J$ respect $\prec$. Moreover, if a binomial $L-T$ belongs to $(I+J)\cap \k[X,Y]$, then $L-T\in I$.
\end{lemma}

\begin{proof}
Note that $G$ is a Gr\"{o}bner basis of $J$ respect $\prec$, and consider $G'=\{g_1,\ldots ,g_h\}\subset \k[X,Y]$ a Gr\"{o}bner basis of $I$. Thus, $S(f,f')$ and $S(g,g')$ have an {\em lcm} representation for every $f,f'\in G$ and $g,g'\in G'$. Let $L-T$ be a binomial in $G'$, $z_i-M_i\in G$, and assume $L\succ T$. Hence, $S(z_i-M_i,L-T) = L(z_i-M_i)-z_i(L-T)= T(z_i-M_i)-M_i(L-T)$; that is to say, $S(z_i-M_i,L-T)$ has an {\em lcm} representation. Therefore, $G\cup G'$ is a Gr\"{o}bner basis of $I+J$.

Consider any $L-T\in I+J$ with $L-T\in \k[X,Y]$. By \cite[Chapter~2, $\S$3, Theorem~3]{CoxLittleOShea} and \cite[Chapter~2, $\S$6, Corollary~2]{CoxLittleOShea}, we have that $L-T=\sum _{i=1}^h f_ig_i $, with $\mathrm{In}_\prec(L-T)\succeq \mathrm{In}_\prec(f_ig_i)$, for $i=1,\ldots ,h$. Hence, every $f_i$ belongs to $\k[X,Y]$.
\end{proof}

A method for computing the ideal $I\cap \k[x_{l+1},\ldots ,x_t]$ (for $t>l\ge 1$) is called the {\em Elimination Theorem}.

\begin{theorem}\cite[Chapter~3]{CoxLittleOShea}
Let $I\subset \k[x_1,\ldots ,x_t]$ be an ideal and let $G$ be a Gr\"{o}bner basis of $I$ with respect to lex order where
$x_1>x_2>\cdots >x_t$. Then, for every $0\le l\le t$, the set $G\cap \k[x_{l+1},\ldots ,x_t]$ is a generating set of the ideal $I\cap \k[x_{l+1},\ldots ,x_t]$. Furthermore, $G\cap \k[x_{l+1},\ldots ,x_t]$ is a Gr\"{o}bner basis of $I\cap \k[x_{l+1},\ldots ,x_t]$.
\end{theorem}

We introduce the semigroup ideal as an important object in this work. A monoid/semigroup is a non-empty set equipped with an associative and commutative binary operation (denoted by $+$), and an identity element. A semigroup is finitely generated if there exists a finite set $A=\{a_1,\ldots ,a_t\}\subset S$ such that $S= \langle A \rangle:= \{\alpha_1 a_1+\cdots +\alpha_t a_t\mid \alpha_1,\ldots ,\alpha_t\in \N \}$ ($\alpha a$ denotes $\sum _{i=1}^\alpha a$). For a field $\k$, $S$ has associated the binomial ideal in $\k[x_1,\ldots ,x_t]$,
$$I_S=\Big\langle \Big\{ X^\alpha - X^\beta \mid \sum_{i=1}^t \alpha_i a_i = \sum_{i=1}^t \beta_i a_i \Big\}\Big\rangle.$$
This ideal is usually called the semigroup ideal of $S$, and it has an important role in studying some properties of the semigroup. Note that $I_S$ codifies the relationships among the elements of $S$. 
Associated to these ideals we have the lattice $M$ of $\Z^t$ generated by the elements $\{\alpha-\beta\mid X^\alpha-X^\beta\in I_S\}$.
We say that $I_S$ is strongly reduced whenever $M\cap \N^t=\{0\}$ (this concept was introduced in \cite{RosalesPedroUrbano}).
Define ${\cal I}_M$ the finitely generated cancellative subsemigroup of $\N^t\times \N^t$ all of whose elements $(\alpha,\beta)$ verify $\alpha-\beta\in M$. Let $\mathcal{A}({\cal I}_M)$ be the minimal system of generators of the subsemigroup ${\cal I}_M$. This minimal generating set can be computed performing the following steps:
\begin{enumerate}
    \item Compute a system of linear homogeneous equations $Ax=0$ of $M$ from its system of generators \cite[Chapter 2]{Rosales}.
    \item The Hilbert basis of $(A|-A)(x,y)=0$ is the set ${\cal A}({\cal I}_M)$. That basis can be computed by using \cite{PisonVigneron}.
\end{enumerate}

When the semigroup $S$ is a subset of $\N$ such that $\N\setminus S$ is finite, $S$ is called a numerical semigroup, and it is finitely generated. In \cite{Rosales_libro}, the authors introduce some algorithms for computing the ideals of numerical semigroups.

\section{Sumset semigroups}\label{sumset semigroups}

Let us begin this section by recalling some standard definitions in semigroup theory. Assume $S$ is a commutative semigroup, $S$ is cancellative if $x+z=y+z$ for some $x,y,z\in S$, implies $x=y$. An element $x\in S$ is a unit if $x+y=0$ for some $y\in S$. The set of units of $S$ is denoted by $\mathcal{U}(S)$. When $S\cap (-S)=\{0\}$, $S$ is named reduced. An atom in $S$ is any non-unit $x\in S$ such that there do not exist two non-units $y,z\in S$ with $x=y+z$. The semigroup is atomic if $S\setminus \mathcal{U}(S)$ is generated by its atoms. The set of atoms of $S$ is denoted by $\mathcal{A}(S)$.

Let $\mathrm{FS}(\N)$ be the set whose elements are the finite non-empty subsets of $\N$. Recall that on $\mathrm{FS}(\N)$, the binary operation $+$ is defined as $A+B=\{a+b \mid a\in A, b\in B\}$ for all $A,B\in \mathrm{FS}(\N)$. The pair $(\mathrm{FS}(\N),+)$ is a commutative monoid with identity element equal to $\{0\}$.
Every finitely generated submonoid of $(\mathrm{FS}(\N),+)$ is called a sumset semigroup.
If $A\in \mathrm{FS}(\N)$ and $\alpha\in\N$, denote by $\alpha \otimes A$ the sumset $\sum_{i=1}^\alpha A$.

The monoid $\mathrm{FS}(\N)$ satisfies the following interesting properties:
\begin{itemize}
\item since $\{1,3\}+\{1,2,3\}=\{1,2,3\}+\{1,2,3\}$, it is non-cancellative;
\item it is a reduced monoid;
\item since $2\otimes \{1,2,4,5\}=2\otimes \{1,2,3,4,5\}$, it is not torsion free;
\item by Proposition 3.2 of \cite{Fan_Tringali}, this monoid is atomic.
\end{itemize}

The operation $\otimes$ has good properties, as shown in the following lemma.

\begin{lemma}
Let $A$, $B$ be in $\mathrm{FS}(\N)$ and $\alpha,\beta\in \N$, then:
\begin{enumerate}
\item $\alpha\otimes(\beta\otimes A)= (\alpha\beta)\otimes A$;
\item $\alpha\otimes (A + B)= \alpha\otimes A + \alpha\otimes B$.
\end{enumerate}
\end{lemma}




Let $S$ be a sumset semigroup minimally generated by $\{A_1,\ldots ,A_t\}$. By definition, the elasticity of a non-unit $A\in S$ is $\rho(A)=\sup \big\{m/n\mid  \exists a_1,\ldots,a_m,b_1,\ldots ,b_n\in \mathcal{A}(S)\text{ with } A=\sum_{i=1}^m a_i=\sum_{i=1}^n b_i\big\}$, and the elasticity of $S$ is $\rho(S)=\sup \{\rho(A)\mid A\in S\setminus\{0\}\}$. 
If there is an element in the monoid whose elasticity “reaches” that of the whole monoid, we say that the monoid has acceptable elasticity.


Note that the ideal associated to $S$ is 
$$I_S=\Big\langle \Big\{ X^\alpha - X^\beta \mid \sum_{i=1}^t \alpha_i \otimes A_i = \sum_{i=1}^t \beta_i \otimes A_i \Big\}\Big\rangle\subset \k[x_1,\ldots ,x_t].$$
Let $\alpha\in M\cap \N^t$ be a non-zero element, thus there exists $\beta \in \N^t$ such that $X^{\alpha+\beta }-X^\beta=X^\beta(X^\alpha-1)\in I_S$, and then $\sum _{i=1}^t (\alpha_i + \beta_i)A_i=\sum _{i=1}^t \beta_i A_i$. Therefore, $\max \sum _{i=1}^t (\alpha_i+ \beta_i)A_i = \sum _{i=1}^t \big(\alpha_i+\beta_i\big)\max A_i> \sum _{i=1}^t \beta_i \max A_i= \max \sum _{i=1}^t \beta_i A_i$, which it is not possible. Hence, $M\cap \N^t$ is $\{0\}$, and the ideal $I_S$ is strongly reduced.

Since $I_S$ is strongly reduced, we have that $S$ is an atomic reduced semigroup with finite elasticity (see Theorem 15 in \cite{atomicos}). Moreover, 
\begin{equation}\label{elasticidad}
\rho (S)= \max \left\{ \frac{\sum_{i=1}^t \alpha_i}{\sum_{i=1}^t\beta_i} \mid (\alpha,\beta)\in \mathcal{A}({\cal I}_M)\right\}.
\end{equation}

If $A=\{a_1<\dots<a_n\}\in\FSN$, then $A=\{a_1\}+\{0,a_2-a_1,\dots,a_n-a_1\}$ (denote $\{0,a_2-a_1,\dots,a_n-a_1\}$ by $\tilde A$). The semigroup $S=\langle A_1,\dots,A_t\rangle$, with $A_i=\{a_{i1}<\dots<a_{it_i}\}$, is a submonoid of $\langle \{a_{11}\},\dots, \{a_{1t}\}, \tilde A_1,\dots, \tilde A_t \rangle$. Trivially, the sumset semigroup $\langle \{a_{11}\},\dots, \{a_{1t}\}\rangle$ is isomorphic to the semigroup $\langle a_{11},\dots, a_{1t}\rangle$, thus $I_{\langle a_1, \ldots ,a_s \rangle}=I_{\langle \{a_{11}\},\dots, \{a_{1t}\} \rangle}$.

\begin{proposition}\label{numerico_y_comienza_cero}
For every $\tilde S=\big\{\{a_1\},\dots,\{a_s\},\tilde A_1,\dots,\tilde A_t\big\}$, with $a_i\neq 0$ and $\min \tilde A_i=0$, we have $I_{\tilde S}= I_{\langle a_1, \ldots ,a_s \rangle} + I_{\langle \tilde A_1, \ldots ,\tilde A_t \rangle}$, where $I_{\langle a_1, \ldots ,a_s \rangle}\subset \k[x_1,\ldots ,x_s]$ and $I_{\langle \tilde A_1, \ldots ,\tilde A_t \rangle}\subset \k[y_1,\ldots ,y_t]$.
\end{proposition}

\begin{proof}
Trivially, $I_{\langle a_1, \ldots ,a_s \rangle}, I_{\langle \tilde A_1, \ldots ,\tilde A_t \rangle}\subset I_{\tilde S}$.

Let $X^\alpha Y^\beta - X^\gamma Y^\delta\in I_{\tilde S}$, then $\sum_{i=1}^s \alpha_i \otimes \{a_i\} + \sum_{i=1}^t\beta_i \otimes \tilde A_i=\sum_{i=1}^s \gamma_i\otimes \{a_i\} + \sum_{i=1}^t\delta_i \otimes \tilde A_i$. Since $\min \tilde A_i=0$, we then have $\sum_{i=1}^s \alpha_i \otimes\{a_i\}=\sum_{i=1}^s \gamma_i\otimes \{a_i\}$, and $\sum_{i=1}^t\beta_i \otimes \tilde A_i = \sum_{i=1}^t\delta_i \otimes \tilde A_i$. That is to say, $X^\alpha-X^\gamma\in I_{\langle a_1, \ldots ,a_s \rangle}$, and $Y^\beta -Y^\delta\in I_{\langle \tilde A_1, \ldots ,\tilde A_t \rangle}$.
Note that $X^\alpha Y^\beta - X^\gamma Y^\delta= Y^\beta(X^\alpha - X^\gamma) + X^\gamma(Y^\beta - Y^\delta )\in I_{\langle a_1, \ldots ,a_s \rangle} + I_{\langle \tilde A_1, \ldots ,\tilde A_t \rangle} \subset \k [x_1,\ldots x_s,y_1\ldots ,y_t]$.
\end{proof}

Since the semigroup $\langle a_1, \ldots ,a_s \rangle$ is isomorphic to a numerical semigroup, there exist algorithms for computing $I_{\langle a_1, \ldots ,a_s \rangle}$. Thus, to compute a presentation of $\tilde S$ we need an algorithm to calculate $I_{\langle \tilde A_1, \ldots , \tilde A_t \rangle}$. In the next sections, we provide some algorithms for computing the ideals of some families of sumset semigroups.

\section{Ideals of a fundamental family of sumset semigroups}\label{ideals}

In this section, we give explicitly the ideals associated with the sumset semigroups generated by the elements $\{0,k\a\}$ and $\{0,k\b\}$, where $\a<\b$ are two positive co-prime integers, and $k\in \N\setminus\{0\}$. These semigroups are key to provide an algorithm to compute the semigroup ideals of more types of sumset semigroups.

Fix $\a<\b$ as two positive co-prime integers and $k\in \N\setminus\{0\}$, and consider the semigroup $\bS=\langle k\a,k\b \rangle$ and the sumset semigroup $\tS$ minimally generated by $\{0,k\a\}$ and $\{0,k\b\}$. We prove that $I_{\tS}\subset \k[x,y]$ is a principal ideal providing its generator. Note that $I_{\tS}\subset I_{\bS}=\langle x^\b-y^\a \rangle$.

\begin{lemma}
Set $x^\alpha y^\beta - x^\gamma y^\delta\in I_{\tS}\setminus\{0\}$. Then, $\alpha\neq \gamma$, $\beta\neq \delta$ and $\alpha \cdot \beta \cdot \gamma \cdot \delta \ge 1 $.
\end{lemma}

\begin{proof}
Note that, since $f=x^\alpha y^\beta - x^\gamma y^\delta\in I_{\tS}$, $\alpha\otimes\{0,k\a\}+\beta \otimes \{0,k\b\}= \gamma\otimes\{0,k\a\}+\delta\otimes \{0,k\b\}$.

Suppose $\alpha=\gamma$, then we have $\alpha k\a +\beta k\b =\max \big\{ \alpha\otimes\{0,k\a\}+\beta \otimes \{0,k\b\} \big\}\ =\max \big\{ \alpha\otimes\{0,k\a\}+\delta\otimes \{0,k\b\} \big\}= \alpha k\a + \delta k\b$, and $\beta = \delta$. Since $f\neq 0$, this is not possible, and therefore $\alpha\neq \gamma$. Analogously, it can be proved that $\beta\neq \delta$.

Suppose $\alpha=0$, then we have $k\b=\min \big\{ \beta\otimes\{0,k\b\}\setminus \{0\} \big\} = \min \big\{ (\gamma\otimes\{0,k\a\}+\delta\otimes \{0,k\b\}) \setminus\{0\} \big\}$. If $\gamma$ is non-zero, then $k\b=\min \big\{ (\gamma\otimes\{0,k\a\}+\delta\otimes \{0,k\b\}) \setminus\{0\} \big\}=k\a$. Therefore, the integers $\gamma$, $\beta$, and $\delta$ are zero and $f=0$. Similarly, $\beta, \gamma, \delta \ge 1 $ can be proved.
\end{proof}

In the sequel, we assume $\alpha>\gamma\ge 1$, and $x^\alpha y^\beta - x^\gamma y^\delta\in I_{\tS}\setminus\{0\}$. Since $I_{\tS}\subset \langle x^\b-y^\a \rangle$, $\alpha \ge \b$ and $\delta \ge \a$.

\begin{lemma}\label{lemma_A}
If $\alpha>\gamma$, then $\delta > \beta$. Additionally, there exists a positive integer $n$ such that $\alpha= n\b+\gamma$ and $\delta =n\a +\beta$.
\end{lemma}

\begin{proof}
Since $x^\alpha y^\beta - x^\gamma y^\delta\in I_{\tS}$, $\alpha\otimes\{0,k\a\}+\beta \otimes\{0,k\b\}= \gamma\otimes\{0,k\a\}+\delta \otimes\{0,k\b\}$, and $\alpha k\a +\beta k\b =\max \big\{ \alpha\otimes\{0,k\a\}+\beta\otimes \{0,k\b\} \big\}\ =\max \big\{ \gamma\otimes\{0,k\a\}+\delta\otimes \{0,k\b\} \big\}= \gamma k\a + \delta k\b$. So, $(\delta-\beta)k\b=(\alpha-\gamma)k\a>0$. Furthermore, $(\alpha-\gamma)/(\delta-\beta) = \b/\a$. Since $\gcd(\a,\b)= 1$, there exist two positive integers $n$ and $m$ such that $\alpha=n\b+\gamma$ and $\delta=m\a + \beta$. From the equality $\alpha k\a +\beta k\b=\gamma k\a + \delta k\b$, 
we deduce that $n=m$.
\end{proof}

\begin{lemma}\label{lemma_B}
Let $x^\alpha y^\beta - x^\gamma y^\delta\in I_{\tS}\setminus\{0\}$. Then, $\gamma\ge \b-1$, $\beta \ge \a-1$, and there is a positive integer $n\in \N$ such that $x^\alpha y^\beta - x^\gamma y^\delta= x^\gamma y^\beta(x^{n\b}-y^{n\a})$.
\end{lemma}

\begin{proof}
Assume $\gamma<\b-1$. Take $(\gamma+1)k\a+\beta k\b$ in $\alpha\otimes\{0,k\a\}+\beta \otimes\{0,k\b\}$ (recall $\alpha>\gamma$). For that element, there exist two integers $i\in [0,\gamma]$ and $j\in [0,\delta]$ such that $(\gamma+1)k\a+\beta k\b= (\gamma-i)k\a+(\delta-j)k\b\in \gamma\{0,k\a\}+\delta \{0,k\b\}$, hence $(1+i)k\a=(\delta-\beta-j)k\b$. Since $\gcd(\a,\b)= 1$, $i+1\ge \b$ and $i\ge \b -1>\gamma$, but $i\le \gamma$, which is a contradiction. Analogously, the fact $\beta \ge \a-1$ can be proved.

By Lemma~\ref{lemma_A}, there exists $n\in \N\setminus \{0\}$ such that
$$x^\alpha y^\beta - x^\gamma y^\delta= x^{n\b+\gamma} y^\beta - x^\gamma y^{n\a +\beta}= x^\gamma y^\beta(x^{n\b}-y^{n\a}).$$
\end{proof}

\begin{theorem}\label{teorema_comienza_cero}
Let $1\le \a< \b$ be two co-prime integers, $k\in\N\setminus\{0\}$, and $\tS$ be the sumset semigroup $\langle\{0,k\a\},\{0,k\b\}\rangle$. The ideal $I_{\tS}\subset \k[x,y]$ is principal and is generated by $x^{\b-1} y^{\a-1}(x^\b-y^\a)$.
\end{theorem}

\begin{proof}
Observe that $x^{\b-1} y^{\a-1}(x^\b-y^\a)= x^{2\b-1} y^{\a-1}-x^{\b-1} y^{2\a-1}$. To prove this theorem, we describe explicitly the sets $A=(2\b-1)\otimes\{0,k\a\}+(\a-1)\otimes\{0,k\b\}$ and $B=(\b-1)\otimes\{0,k\a\}+(2\a-1)\otimes\{0,k\b\}$ associated with the monomials $x^{2\b-1} y^{\a-1}$ and $x^{\b-1} y^{2\a-1}$ (respectively), to achieve $A=B$.

Note that the first set $A=(2\b-1)\otimes\{0,k\a\}+(\a-1)\otimes\{0,k\b\}$ is equal to
\begin{equation*}
\begin{split}
\{0,k\a,\ldots ,(2\b-1)k\a\}+\{0,k\b,\ldots ,(\a-1)k\b\}=\\
= k\Big(\{0,\a,\ldots ,(2\b-1)\a\}\cup \{0,\b,\ldots ,(\a-1)\b\}\cup \\
\{\a+\b,\ldots ,\a+(\a-1)\b\}\cup \{2\a+\b,\ldots ,2\a+(\a-1)\b\}\cup
\cdots \\
\cup \{(2\b-1)\a+\b,\ldots ,(2\b-1)\a+(\a-1)\b\}\Big)\\
= k\Big(\{0,\a,\ldots ,(\b-1)\a,\b\a,(\b+1)\a ,\ldots ,(2\b-1)\a\}\cup \\
\{0,\b,\ldots ,(\a-1)\b\}\cup \{\a+\b,\ldots ,\a+(\a-1)\b\}\cup
\cdots \\
\cup \{(\b-1)\a+\b,\ldots ,(\b-1)\a+(\a-1)\b\}\cup \\
\{\b\a+\b,\ldots ,\b\a+(\a-1)\b\}\cup \\
\{(\b+1)\a+\b,\ldots ,(\b+1)\a+(\a-1)\b\}\cup 
\cdots \\
\cup \{(2\b-1)\a+\b,\ldots ,(2\b-1)\a+(\a-1)\b\}\Big).\\
\end{split}
\end{equation*}
We denote $C_1=\{0,\a,\ldots ,(\b-1)\a\}$, $C_2=\{0,\b,\ldots ,(\a-1)\b\}$,
$C_3=\cup_{i=1}^{\b-1}\{i\a+\b,\ldots ,i\a+(\a-1)\b\}$, $C_4=\{\a\b\}\cup \{\b\a+\b,\ldots ,\b\a+(\a-1)\b\}$, $C_5=\{(\b+1)\a ,\ldots ,(2\b-1)\a\}$, and $C_6=\cup_{i=\b+1}^{2\b-1} \{i\a+\b,\ldots ,i\a+(\a-1)\b\}$. The set $A$ is the union $\cup _{i=1}^6 kC_i$.

The set $B=(\b-1)\otimes\{0,k\a\}+(2\a-1)\otimes\{0,k\b\}$ is
\begin{multline*} 
\{0,k\a,\ldots ,(\b-1)k\a\}+\{0,k\b,\ldots ,(2\a-1)k\b\}=\\
= k\Big(\{0,\a,\ldots ,(\b-1)\a\}\cup
\{0,\b,\ldots ,(\a-1)\b,\a\b,\ldots ,(2\a-1)\b\}\cup \\
\{\a+\b,\ldots ,\a+(\a-1)\b,\a+\a\b,\ldots,\a+(2\a-1)\b\}\cup \cdots \\
\cup \{(\b-1)\a+\b,\ldots , (\b-1)\a+(\a-1)\b,(\b-1)\a+\a\b ,\ldots, (\b-1)\a+(2\a-1)\b\}\Big)\\
= \cup _{i=1}^6 kC_i.
\end{multline*}
Thus, $A=B$, and $x^{\b-1} y^{\a-1}(x^\b-y^\a)\in I_S$.

To finish the proof, we use Lemma~\ref{lemma_B}. If $n=1$, then $x^\alpha y^\beta - x^\gamma y^\delta= x^{\gamma-\b+1} y^{\beta-\a+1} x^{\b-1} y^{\a-1}(x^\b-y^\a)$.
In case $n>1$, by factorizing the binomial $x^{n\b}-y^{n\a}$, we obtain
\begin{multline*}
x^\alpha y^\beta - x^\gamma y^\delta = \\ x^{\gamma-\b+1} y^{\beta-\a+1}(x^{(n-1)\b}+x^{(n-2)\b}y+\cdots + xy^{(n-2)\a}+ y^{(n-1)\a}) { x^{\b-1} y^{\a-1}(x^\b-y^\a)}.
\end{multline*}
In any case, $I_{\tS}=\big\langle x^{\b-1} y^{\a-1}(x^\b-y^\a)\big\rangle$.
\end{proof}

\begin{corollary}\label{corollary_A}
Let $a_1,\ldots ,a_s$ be positive integers, and $S$ be the sumset semigroup generated by $\big\{\{a_1\},\dots,\{a_s\},\{0,k\a\},\{0,k\b\}\big\}$. Then,
$$I_{S}=I_{\langle a_1, \ldots ,a_s \rangle}+ \big\langle x^{\b-1} y^{\a-1}(x^\b-y^\a) \big\rangle\subset \k [x_1,\ldots ,x_s,x,y].$$
\end{corollary}

\begin{proof}
From Proposition~\ref{numerico_y_comienza_cero} and Theorem~\ref{teorema_comienza_cero}, we obtain the result.
\end{proof}

\section{Computing the ideals of sumset semigroups}\label{computing}

The aim of this section is to determine an algorithm for computing the ideals associated with some families of sumset semigroups. As in the previous section, we consider two positive co-prime integers $\a<\b$, $k\in \N\setminus\{0\}$, the semigroup $\bS=\langle k\a,k\b \rangle$, and the sumset semigroup $\tS=\langle \{0,k\a\},\{0,k\b\}\rangle$.

For any two non-negative integers $n$ and $m$, $A_{nm}$ denotes $\big\{\alpha k\a+\beta k\b \mid \alpha \in \{0,\ldots ,n\},\, \beta\in \{0,\ldots ,m\} \big\}=n\otimes k\{0,\a\}+ m\otimes k\{0,\b\}$.

\begin{theorem}\label{main_theorem}
Let $\{(n_i,m_i)\mid n_i,m_i\in \N,\,n_i+m_i>0,\, i=1,\ldots ,t\}$ be a non-empty subset of $\N^2$, $b_1,\ldots ,b_p,a_1,\ldots ,a_s\in \N\setminus \{0\}$ with $s\le t$, and consider
$S$ the sumset semigroup generated by $$\big\{\{b_1\},\ldots ,\{b_p\}, \{a_1\}+A_{n_1m_1},\ldots, \{a_s\}+A_{n_sm_s},A_{n_{s+1}m_{s+1}},\ldots ,A_{n_tm_t}\big\},$$
and the sumset semigroup $S'=\big\langle \{b_1\},\ldots ,\{b_p\}, \{a_1\},\ldots, \{a_s\}, \{0,k\a\},\{0,k\b\} \big\rangle$.
Then, $I_S\in \k [x_1,\ldots ,x_p,z_1,\ldots , z_t]$ is
\begin{multline*}
\Big(I_{S'} + \big\langle z_1-w_1x^{n_1}y^{m_1},\ldots , z_s-w_sx^{n_s}y^{m_s}, z_{s+1}-x^{n_{s+1}}y^{m_{s+1}},\ldots ,z_t-x^{n_t}y^{m_t} \big\rangle\Big)\\ \bigcap \k [x_1,\ldots ,x_p,z_1,\ldots ,z_t],
\end{multline*}
with $I_{S'}\subset \k[x_1,\ldots ,x_p, w_1,\ldots ,w_s ,x,y]$.
\end{theorem}

\begin{proof}
Denote $J$ to the ideal $$\big\langle z_1-w_1x^{n_1}y^{m_1},\ldots , z_s-w_s x^{n_s}y^{m_s}, z_{s+1}-x^{n_{s+1}}y^{m_{s+1}},\ldots ,z_t-x^{n_t}y^{m_t} \big\rangle.$$

Any monomials $Z^\beta$ and $Z^\delta$ can be rewritten as follows. Denote $f_i=\prod_{k=i}^t z_i^{\beta_j}$, $\hat f_i=\prod_{k=i}^t z_i^{\delta_j}$, $g_i^j=\prod_{k=i}^j(w_kx^{n_k}y^{m_k})^{\beta_k}$,
$\hat g_i^j=\prod_{k=i}^j(w_kx^{n_k}y^{m_k})^{\delta_k}$, $h_i^j=\prod_{k=i}^j(x^{n_k}y^{m_k})^{\beta_k}$, $\hat h_i^j=\prod_{k=i}^j(x^{n_k}y^{m_k})^{\delta_k}$, and suppose $\beta_0=\delta_0=0$, then we have that
$$Z^\beta = z_1^{\beta_1}\cdots z_t^{\beta_t} = \sum_{i=0}^{s-1} f_{i+2} g_0^i\big(z_{i+1}^{\beta_{i+1}}-g_{i+1}^{i+1}\big)+
\sum_{i=s+1}^t f_{i+1}g_1^s h_{s+1}^{i-1}\big(z_{i}^{\beta_{i}}-h_{i}^{i}\big)+g_1^sh_{s+1}^t$$
and
$$Z^\delta = z_1^{\delta_1}\cdots z_t^{\delta_t} = \sum_{i=0}^{s-1} \hat f_{i+2} \hat g_0^i\big(z_{i+1}^{\delta_{i+1}}-\hat g_{i+1}^{i+1}\big)+
\sum_{i=s+1}^t \hat f_{i+1}\hat g_1^s \hat h_{s+1}^{i-1}\big(z_{i}^{\delta_{i}}-\hat h_{i}^{i}\big)+\hat g_1^s\hat h_{s+1}^t.$$
Denote $F=X^\alpha g_1^sh_{s+1}^t -X^\gamma \hat g_1^s\hat h_{s+1}^t$.

Since any binomial $u^n-v^n$ is equal to $(u-v)(u^{n-1}+u^{n-2}v+\cdots +uv^{n-2}+v^{n-1})$, the binomials $z_i^n-(w_ix^{n_i}y^{m_i})^n$ and $z_j^n-(x^{n_j}y^{m_j})^n$ belong to $J$ for every non-negative integer $n$ and all $i=1,\ldots ,s$ and $j=s+1, \ldots ,t$.

Let $X^\alpha Z^\beta-X^\gamma Z^\delta$ be a binomial belonging to $I_{S'} + J$, so
$F= X^\alpha Z^\beta-X^\gamma Z^\delta + q\in I_{S'}+J,$ with $q\in J$.
By Lemma~\ref{para_teorema}, we have $F\in I_{S'}$, hence
\begin{multline*}
\sum_{i=1}^p \alpha_i \otimes \{b_i\} + \sum_{i=1}^s \beta_i\otimes \{a_i\} +\sum_{i=1}^t \Big( \beta_i \otimes \big(n_i\otimes \{0,k\a\}\big) + \beta_i\otimes \big(m_i \otimes \{0,k\b\}\big)\Big)=\\
\sum_{i=1}^p \gamma_i \otimes \{b_i\} + \sum_{i=1}^s \delta_i\otimes \{a_i\} +\sum_{i=1}^t \Big( \delta_i \otimes \big(n_i\otimes \{0,k\a\}\big) + \delta_i\otimes \big(m_i \otimes \{0,k\b\}\big)\Big).
\end{multline*}
Therefore,
\begin{multline*}
\sum_{i=1}^p \alpha_i \otimes \{b_i\} + \sum_{i=1}^s \beta_i\otimes \big(\{a_i\} + A_{n_im_i}\big) +
\sum_{i=s+1}^t \beta_i A_{n_im_i}=\\
\sum_{i=1}^p \gamma_i \otimes \{b_i\} + \sum_{i=1}^s \delta_i\otimes \big(\{a_i\} + A_{n_im_i}\big) +
\sum_{i=1}^t \delta_i A_{n_im_i},
\end{multline*}
and $X^\alpha Z^\beta-X^\gamma Z^\delta\in I_{S}$.

Analogously, if $X^\alpha Z^\beta-X^\gamma Z^\delta\in I_S$, then $F\in I_{S'}$, and $X^\alpha Z^\beta-X^\gamma Z^\delta\in I_{S'} + J$. This completes the proof.

\end{proof}

The above proof can also be done by using \cite[Proposition~4]{presentaciones}. In our proof, we employ the language of polynomials, ideals and Gröbner bases, avoiding congruences.

From Theorem~\ref{main_theorem}, we obtain an algorithm (Algorithm~\ref{algoritmo_ideal_familia}) for computing the ideal of the sumset semigroup generated by $\big\{\{b_1\},\ldots ,\{b_p\}, \{a_1\}+A_{n_1m_1},\ldots, \{a_s\}+A_{n_sm_s},A_{n_{s+1}m_{s+1}},\ldots ,A_{n_tm_t}\big\}$.

\begin{algorithm}[H]
\BlankLine
\KwData{$\big\{\{b_1\},\ldots ,\{b_p\}, \{a_1\}+A_{n_1m_1},\ldots, \{a_s\}+A_{n_sm_s},A_{n_{s+1}m_{s+1}},\ldots ,A_{n_tm_t}\big\}$, the generating set of $S$ .}
\KwResult{$\cal G$, a generating set of the semigroup ideal of $S$.}
\BlankLine
\Begin
{ $S'\leftarrow \big\{\{b_1\},\ldots ,\{b_p\}, \{a_1\},\ldots, \{a_s\}, \{0,k\a\},\{0,k\b\}\big\}$\;

$S_1\leftarrow \big\{\{b_1\},\ldots ,\{b_p\}, \{a_1\},\ldots, \{a_s\}\big\}$\;

$G_1\leftarrow$ a generating set of $I_{\langle S_1 \rangle}\subset \k[X, W]$\;

$G' \leftarrow G_1 \sqcup \{x^{\b-1} y^{\a-1}(x^\b-y^\a)\}\subset \k[X, W,x,y]$, it is a generating set of $I_{\langle S' \rangle}$ (Corollary~\ref{corollary_A})\;

$G_2\leftarrow G'\sqcup \{ z_1-w_1x^{n_1}y^{m_1},\ldots , z_s-w_sx^{n_s}y^{m_s}, z_{s+1}-x^{n_{s+1}}y^{m_{s+1}},\ldots ,z_t-x^{n_t}y^{m_t} \} \subset \k [X,W,x,y,Z]$\;

$G_3 \leftarrow$ a Gr\"{o}bner basis of $G_2$ respect to a monomial order with $x,y,w_q>x_i$ and $x,y,w_q>z_j$ for every $i=1,\ldots ,p$, $j=1,\ldots ,t$, and $q=1,\ldots ,s$\;

${\cal G}\leftarrow \{X^\alpha Z^\beta-X^\gamma Z^\delta \mid X^\alpha Z^\beta-X^\gamma Z^\delta \in G_3\}$, it is a generating set (Gr\"{o}bner basis) of $I_{S}$\;

\Return $\cal G$\;
}
\caption{Computation of $I_S$.}\label{algoritmo_ideal_familia}
\end{algorithm}

We show how this algorithm works with an example.

\begin{example}\label{ejemplo_base}
Let $S$ be the sumset semigroup generated by \[\big\{\{3\},\{4\},\{6,12\},\{7,10,13\},\{0,3,6,9\}
\big\}.\]
Then, from the first steps of Algorithm~\ref{algoritmo_ideal_familia},
\begin{itemize}
\item $S'=\big\{\{3\},\{4\},\{6\},\{7\},\{0,3\},\{0,6\}\big\}$,
\item $S_1=\big\{\{3\},\{4\},\{6\},\{7\}\big\}$.
\end{itemize}

If we compute a generating set of the ideal of $S_1$, then we get the following one,
$G_1=\big\{w_1^7-w_2^6,w_2^2 x_2-w_1^3,w_1^4 x_2-w_2^4,w_1x_2^2-w_2^2,x_2^3-w_1^2,w_2 x_1-w_1 x_2,w_1^2 x_1-w_2 x_2^2,x_1x_2-w_2,x_1^2-w_1\big\}$.
Therefore, $G'=G_1\cup \big\{x (x^2-y)\big\}$ and $G_2=G'\cup \big\{z_1-w_1 y,z_2-w_2 x^2,z_3-x^3\big\}.$

Now, we compute a Gröbner basis of $G_2$ respect to the lexicographical order where $x>y>w_i>x_j>z_k$ for all $i$, $j$, and $k$, and we obtain
\begin{equation*}
\begin{split}
G_3=\big\{&z_1^{21} z_3-z_2^{18} z_3^3, z_1^{21} z_2-z_2^{19} z_3^2, x_2z_2^{14} z_3^3-z_1^{17} z_3, x_2 z_2^{15} z_3^2-z_1^{17} z_2,\\&
x_2 z_1^4 z_3-z_2^4 z_3, x_2 z_1^4 z_2-z_2^5, x_2^2 z_2^{10} z_3^3-z_1^{13}z_3, x_2^2 z_2^{11} z_3^2-z_1^{13} z_2, \\&
x_2^3 z_2^6 z_3^3-z_1^9z_3, x_2^3 z_2^7 z_3^2-z_1^9 z_2, x_2^4 z_2^2 z_3^3-z_1^5 z_3, x_2^4z_2^3 z_3^2-z_1^5 z_2, \\&
x_2^5 z_3^2-z_1 z_2^2, x_1 z_2 z_3-x_2 z_1z_3, x_1 z_2^2-x_2 z_1 z_2, x_1 z_1^3 z_3-z_2^3 z_3, \\&
x_1 z_1^3z_2-z_2^4, x_1 x_2^4 z_3^2-z_1^2 z_2, x_1^2 z_1^2 z_3-x_2 z_2^2z_3, x_1^2 z_1^2 z_2-x_2 z_2^3, \\&
x_1^2 x_2^3 z_3^3-z_1^3 z_3, x_1^3 z_1z_3-x_2^2 z_2 z_3, x_1^3 z_1 z_2-x_2^2 z_2^2, x_1^3 x_2^3z_3^2-z_2^3,\\&
x_1^4-x_2^3,w_2-x_1 x_2, w_1-x_1^2, y z_2^2-x_1^2 x_2^2z_3^2, y z_1 z_2-x_1^3 x_2 z_3^2,\\&
y z_1^2 z_3-x_2^3 z_3^3, y x_2^2 z_2z_3-x_1 z_1^2 z_3,y x_2^2 z_1 z_3-z_2^2 z_3, y x_2^3-x_1^2 z_1,\\&
y x_1x_2 z_3-z_2 z_3, y x_1 x_2 z_2-z_2^2, y x_1^2-z_1, y^2 z_2-x_1 x_2z_3^2, y^2 z_1 z_3-x_1^2 z_3^3,\\&
y^3 z_3-z_3^3, x z_3^2-y^2 z_3, x z_2-x_1x_2 z_3, x z_1-x_1^2 z_3, x x_2^4 z_3-x_1 z_1 z_2,\\&
x x_1 x_2 z_3-y z_2, x y-z_3, x^2 z_3-y z_3, x^2 x_2^4-x_1^3 z_2, x^2 x_1 x_2-z_2,x^3-z_3\big\}.
\end{split}
\end{equation*}
Finally, the output of the algorithm is 
\begin{equation*}
\begin{split}
\big\{& z_1^{21} z_3-z_2^{18} z_3^3,z_1^{21} z_2-z_2^{19} z_3^2,x_2z_2^{14} z_3^3-z_1^{17} z_3,
x_2 z_2^{15} z_3^2-z_1^{17} z_2,\\&
x_2 z_1^4z_3-z_2^4 z_3,x_2 z_1^4 z_2-z_2^5,
x_2^2 z_2^{10} z_3^3-z_1^{13}z_3,x_2^2 z_2^{11} z_3^2-z_1^{13} z_2,\\&
x_2^3 z_2^6 z_3^3-z_1^9z_3,x_2^3 z_2^7 z_3^2-z_1^9 z_2,
x_2^4 z_2^2 z_3^3-z_1^5 z_3,x_2^4z_2^3 z_3^2-z_1^5 z_2,\\&
x_2^5 z_3^2-z_1 z_2^2,x_1 z_2 z_3-x_2 z_1z_3,
x_1 z_2^2-x_2 z_1 z_2,x_1 z_1^3 z_3-z_2^3 z_3,
x_1 z_1^3z_2-z_2^4,\\&
x_1 x_2^4 z_3^2-z_1^2 z_2,x_1^2 z_1^2 z_3-x_2 z_2^2z_3,
x_1^2 z_1^2 z_2-x_2 z_2^3,x_1^2 x_2^3 z_3^3-z_1^3 z_3,\\&
x_1^3 z_1z_3-x_2^2 z_2 z_3,
x_1^3 z_1 z_2-x_2^2 z_2^2,
x_1^3 x_2^3z_3^2-z_2^3,x_1^4-x_2^3\big\}.
\end{split}
\end{equation*}
Since the binomial $x_1 z_2 z_3-x_2 z_1z_3\in I_S$, $\{3\}+\{7,10,13\}+\{0,3,6,9\}=\{4\}+\{6,12\}+\{0,3,6,9\}$, but
$\{3\}+\{7,10,13\}\neq\{4\}+\{6,12\}$, the semigroup $S$ is non-cancellative.
\end{example}

The following example introduces an algorithm to obtain an expression for an integer set as a sum of other given integer sets, if possible. In particular, the $i$-fold sumset of a set is studied.

\begin{example}
We now use the above presentation of $S$ to check whether the element $i\otimes \{7,10,13\}$ can be expressed in terms of the other generators of the semigroup $S$. We compute the Gr\"obner basis with respect to the order given by the matrix 
$$A=
\left(
\begin{array}{ccccc}
0 & 0 & 0 & 1 & 0 \\
1 & 1 & 1 & 0 & 1 \\
1 & 0 & 0 & 0 & 0 \\
0 & 1 & 0 & 0 & 0 \\
0 & 0 & 1 & 0 & 0 \\
\end{array}
\right),
$$
and we obtain the set 
\begin{multline*}
G_A=\big\{
x_1^4-x_2^3,x_1^2 x_2^3 z_3^3-z_1^3 z_3,x_2^6 z_3^3-x_1^2 z_1^3 z_3,z_1^2 z_2-x_1 x_2^4 z_3^2,\\x_1 z_2 z_3-x_2 z_1 z_3,
x_2^2 z_2 z_3-x_1^3 z_1 z_3,z_1 z_2^2-x_2^5 z_3^2,x_1 z_2^2-x_2 z_1 z_2,\\x_2 z_2^2 z_3-x_1^2 z_1^2 z_3,x_2^2 z_2^2-x_1^3 z_1 z_2,z_2^3-x_1^3 x_2^3 z_3^2
\big\}.
\end{multline*}
In Table~\ref{table1}, we show some elements $z_2^i$ that when reduced with respect to the basis $G_A$ are expressed by using only the variables $x_1$, $x_2$, $z_1$, and $z_3$, and the expression of $i\otimes\{7,10,13\}$ in terms of the elements of the set $\big\{\{3\},\{4\},\{6,12\},\{0,3,6,9\}\big\}$.

\begin{table}[h]
\centering
\begin{tabular}{r|l|l}
$i$ & Reduction of $z_2^i$ & $i\otimes \{7,10,13\}=$\\ \hline
$3$ & $x_1^3 x_2^3 z_3^2$ & $3\otimes\{3\}+3\otimes\{4\}+2\otimes\{0,3,6,9\}$\\
$4$ & $x_1^2 x_2^4 z_1 z_3^2$ & $2\otimes\{3\}+4\otimes\{4\}+1\otimes\{0,6,12\}+2\otimes\{0,3,6,9\}$ \\
$5$ & $x_1 x_2^5 z_1^2 z_3^2$ & $1\otimes\{3\}+5\otimes\{4\}+2\otimes\{0,6,12\}+2\otimes\{0,3,6,9\}$\\
$6$ & $x_2^6 z_1^3 z_3^2$ & $6\otimes\{4\}+3\otimes\{6,12\}+2\otimes\{0,3,6,9\}$\\
$7$ & $x_1^3 x_2^4 z_1^4 z_3^2$ & $3\otimes\{3\}+4\otimes\{4\}+4\otimes\{0,6,12\}+2\otimes\{0,3,6,9\}$\\
$8$ & $x_1^2 x_2^5 z_1^5 z_3^2$ & $2\otimes\{3\}+5\otimes\{4\}+5\otimes\{0,6,12\}+2\otimes\{0,3,6,9\}$\\
\end{tabular}
\caption{Expressions of some elements $i\otimes\{7,10,13\}$ in terms of the subset of generators $\big\{\{3\},\{4\},\{6,12\},\{0,3,6,9\}\big\}$.}
\label{table1}
\end{table}

In general, the reduction of $z_2^i$ with respect to $G_A$ is 
\[
x_1^{(2-i)\mod 4}\,
x_2^{\lfloor \frac{i+1}4 \rfloor+2+\big((i-3)\mod 4\big) }
z_1^{i-3}
z_3^{2}.
\] 
Therefore, for every $i\geq 3$,
\begin{multline*}
i\otimes \{7,10,13\}=
\big({(2-i)\mod 4}\big)\otimes\{3\}+\\
\Big({\Big\lfloor \frac{i+1}4 \Big\rfloor+2+\big((i-3)\mod 4\big) }\Big)\otimes\{4\}+
(i-3)\otimes\{6,12\}+
2\otimes\{0,3,6,9\}.
\end{multline*}

\end{example}

The last examples are dedicated to study the elasticity of a sumset semigroup.

\begin{example}\label{ex14}
Again, consider the semigroup $S$ given in example \ref{ejemplo_base}. From its ideal, we compute a generating set of its associated lattice $M$,
\begin{multline*}\big\{
 \{0,0,21,-18,-2\}, \{0,1,-17,14,2\}, \{0,1,4,-4,0\}, \{0,2,-13,10,2\},\\
 \{0,3,-9,6,2\}, \{0,4,-5,2,2\}, \{0,5,-1,-2,2\}, \{1,-1,-1,1,0\},\\ 
 \{1,0,3,-3,0\}, \{1,4,-2,-1,2\}, \{2,-1,2,-2,0\}, \{2,3,-3,0,2\},\\ 
 \{3,-2,1,-1,0\}, \{3,3,0,-3,2\}, \{4,-3,0,0,0\} 
\big\},\end{multline*}
and its system of linear homogeneous equations,
$$Ax=\left(
\begin{array}{ccccc}
 -3 & -4 & 2 & 1 & 12 \\
 -6 & -8 & 2 & 0 & 21 \\
\end{array}
\right)x=0.$$
We already know that $S$ is strongly reduced, but this fact is far for being clear from the above equations. We can check it by computing with Normaliz \cite{normaliz} its Hilbert basis:
\begin{verbatim}
    >>> c1=Cone(equations=[[-3,-4,2,1,12],[-6,-8,2,0,21]])
    >>> c1.HilbertBasis()
    []
\end{verbatim}
Since the above output is the empty list, we have $M\cap \N^5=\{0\}$.
The Hilbert basis of $(A\mid -A)(x,y)=0$ has $109$ elements: 
{\tiny 
\begin{equation*}
\begin{split}
    HB= \big\{&
    \{0,1,4,0,0,0,0,0,4,0\},\{1,0,3,0,0,0,0,0,3,0\},\{0,0,0,1,0,0,0,0,1,0\},
    \{0,0,21,0,0,0,0,0,18,2\},\\& \{0,21,0,0,8,0,0,0,12,0\},\{0,0,0,12,0,0,21,0,0,8\},\{0,0,1,0,0,0,0,1,0,0\},\{0,0,0,0,1,0,0,0,0,1\},\\&
    \{0,0,0,18,2,0,0,21,0,0\},\{1,0,0,15,2,0,0,18,0,0\},\{0,1,0,14,2,0,0,17,0,0\},\{1,0,0,0,0,1,0,0,0,0\},\\&
    \{0,0,0,10,0,0,16,1,0,6\},\{0,1,1,0,0,1,0,0,1,0\},\{0,1,0,0,0,0,1,0,0,0\},\{0,16,1,0,6,0,0,0,10,0\},\\&
    \{1,15,0,0,6,0,0,0,9,0\},\{0,0,0,9,0,1,15,0,0,6\},\{2,0,0,12,2,0,0,15,0,0\},\{2,0,2,0,0,0,1,0,2,0\},\\&
    \{0,0,6,0,0,0,9,0,0,4\},\{0,0,7,0,0,1,8,0,1,4\},\{1,1,0,11,2,0,0,14,0,0\},\{0,0,18,0,0,1,0,0,15,2\},\\&
    \{0,0,17,0,0,0,1,0,14,2\},\{0,2,0,10,2,0,0,13,0,0\},\{3,0,0,9,2,0,0,12,0,0\},\{0,0,0,8,0,0,11,2,0,4\},\\&
    \{2,1,0,8,2,0,0,11,0,0\},\{0,11,2,0,4,0,0,0,8,0\},\{2,0,3,0,0,0,6,0,0,2\},\{0,0,0,7,0,1,10,1,0,4\},\\&
    \{0,0,1,2,0,4,2,0,0,2\},\{1,10,1,0,4,0,0,0,7,0\},\{0,0,14,0,0,1,1,0,11,2\},\{0,0,15,0,0,2,0,0,12,2\},\\&
    \{0,0,13,0,0,0,2,0,10,2\},\{1,2,0,7,2,0,0,10,0,0\},\{2,9,0,0,4,0,0,0,6,0\},\{0,0,0,6,0,2,9,0,0,4\},\\&
    \{0,0,2,1,0,5,1,0,0,2\},\{0,10,0,0,4,1,0,5,1,0\},\{0,0,1,2,0,0,5,0,0,2\},\{4,0,0,6,2,0,0,9,0,0\},\\&
    \{1,8,0,1,4,0,0,7,0,0\},\{0,0,3,0,0,6,0,0,0,2\},\{0,3,0,6,2,0,0,9,0,0\},\{1,0,4,0,0,0,5,0,1,2\},\\&
    \{3,1,0,5,2,0,0,8,0,0\},\{0,0,10,0,0,1,2,0,7,2\},\{0,0,11,0,0,2,1,0,8,2\},\{0,0,12,0,0,3,0,0,9,2\},\\&
    \{0,0,9,0,0,0,3,0,6,2\},\{0,9,0,0,4,0,0,6,0,0\},\{3,0,1,0,0,0,2,0,1,0\},\{0,0,0,6,0,0,6,3,0,2\},\\&
    \{0,0,2,1,0,1,4,0,0,2\},\{2,2,0,4,2,0,0,7,0,0\},\{1,0,5,1,0,0,10,0,0,4\},\{4,0,0,0,0,0,3,0,0,0\},\\&
    \{0,3,0,0,0,4,0,0,0,0\},\{0,0,6,0,0,5,0,0,3,2\},\{0,0,5,0,0,4,1,0,2,2\},\{0,0,4,0,0,3,2,0,1,2\},\\&
    \{0,0,3,0,0,2,3,0,0,2\},\{9,0,0,0,2,0,2,2,1,0\},\{0,0,0,5,0,5,2,2,0,2\},\{5,0,0,3,2,0,0,6,0,0\},\\&
    \{0,0,5,0,0,0,4,0,2,2\},\{0,0,9,0,0,4,0,0,6,2\},\{0,0,8,0,0,3,1,0,5,2\},\{0,0,6,0,0,1,3,0,3,2\},\\&
    \{0,0,7,0,0,2,2,0,4,2\},\{1,0,0,1,0,0,1,1,0,0\},\{0,0,0,5,0,1,5,2,0,2\},\{0,7,0,0,2,3,0,2,1,0\},\\&
    \{1,3,0,3,2,0,0,6,0,0\},\{0,2,2,1,0,9,0,0,0,2\},\{0,2,0,1,0,3,0,1,0,0\},\{0,1,0,2,0,2,0,2,0,0\},\\&
    \{0,0,0,3,0,1,0,3,0,0\},\{0,6,3,0,2,0,0,0,6,0\},\{5,2,2,0,2,0,0,0,5,0\},\{8,0,0,0,2,0,1,1,2,0\},\\&
    \{0,0,0,4,0,0,1,4,0,0\},\{0,0,0,4,0,6,1,1,0,2\},\{4,1,0,2,2,0,0,5,0,0\},\{0,1,1,2,0,8,0,0,0,2\},\\&
    \{6,1,1,0,2,0,0,0,4,0\},\{1,5,2,0,2,0,0,0,5,0\},\{0,5,0,0,2,0,0,1,2,0\},\{1,4,0,0,2,0,0,2,1,0\},\\&
    \{2,3,0,0,2,0,0,3,0,0\},\{3,0,2,1,0,0,7,0,0,2\},\{0,0,0,4,0,2,4,1,0,2\},\{0,4,0,2,2,0,0,5,0,0\},\\&
    \{0,5,0,1,2,1,0,4,0,0\},\{0,6,0,0,2,2,0,3,0,0\},\{0,0,0,3,0,7,0,0,0,2\},\{7,0,0,0,2,0,0,0,3,0\},\\&
    \{4,2,0,0,2,0,0,1,2,0\},\{5,1,0,0,2,0,0,2,1,0\},\{6,0,0,0,2,0,0,3,0,0\},\{2,4,1,0,2,0,0,0,4,0\},\\&
    \{1,0,0,3,0,0,6,0,0,2\},\{0,6,0,0,2,1,0,0,3,0\},\{3,2,0,1,2,0,0,4,0,0\},\{0,0,0,3,0,3,3,0,0,2\},\\&
    \{3,3,0,0,2,0,0,0,3,0\}
    \big\}
\end{split}
\end{equation*}
}%

Now, using the formula (\ref{elasticidad}), we conclude that the elasticity of $S$ is $3$.

To know if $S$ has acceptable elasticity, we use Algorithm 28 of \cite{atomicos} which is implemented in 
\url{https://github.com/D-marina/CommutativeMonoids/blob/master/Sumsetssemigroups/sumsetSemigroups.ipynb} of \cite{commutative}. 
Running the commands, 
\begin{verbatim}
    >>> gb14=computationIS([3],[4],[6,12],[7,10,13],[0,3,6,9])
    >>> hasAcceptableElasticity(gb14, debug=True)
    ... 
	    We compute a Groebner basis with the lex ordering of the variables
	    [x2, z1, x1, z2, z3]
	    GroebnerBasis([-x1**4 + x2**3, x1**4*x2**2*z3**2 - z1*z2**2,
	    -x1**3*z1*z2 + x2**2*z2**2, -x1**3*z1*z3 + x2**2*z2*z3,
	    -x1*z2**2 + x2*z1*z2, -x1*z2*z3 + x2*z1*z3,
	    x1**5*x2*z3**2 - z1**2*z2, -x1**2*z1**2*z2 + x2*z2**3,
	    -x1**2*z1**2*z3 + x2*z2**2*z3, -x1**6*z2*z3**2 + z1**3*z2,
	    -x1**6*z3**3 + z1**3*z3, x1**7*z3**2 - z2**3],
	    x2, z1, x1, z2, z3, domain='ZZ', order='lex')
	    Once removed the variables [x2, z1], we obtain
	    GroebnerBasis([x1**7*z3**2 - z2**3], x1, z2, z3, domain='ZZ',
	    order='lex')
    ... 
    True
\end{verbatim}%
we obtain that the monoid has acceptable elasticity. 
Moreover, from the above output, we see that the binomial $x_1^7 x_3^2-z_2^3$ belongs to $I_S$. 
Since the quotient of the addition of the exponents of these two monomials is $(7+2)/3=3=\rho(S)$, the element $7\otimes \{3\}+2\otimes\{0,3,6,9\}=3\otimes \{7,10,13\}$ reaches the elasticity.


\end{example}

We see now an example of sumset semigroup without acceptable elasticity.

\begin{example}\label{ex15}
Let $S$ be the semigroup 
$$\big\langle \{0,3\},\{0,4\},\{7,10,11,13,14,15,17,18,19,1,22,25\} \big\rangle.$$
Analogously as in the preceding example, we use function {\tt hasAcceptableElasticity} to check if $S$ has acceptable elasticity.
The display output of this function shows some steps of Algorithm 28 of \cite{atomicos} and returns {\tt False}, that is, the semigroup has not acceptable elasticity.
\begin{verbatim}
    >>> gb15=computationIS([0,3],[0,4],
            [7,10,11,13,14,15,17,18,19,1,22,25])
    >>> hasAcceptableElasticity(gb15, debug=True)  
    Positive cone of M (the semigroup is strongly reduced): []
    Equations of M
    [[3, 4, 0], [0, 0, 1]]
    Matrix (A|-A) (equations of M \cap N^n, n=5):
    [[3, 4, 0, -3, -4, 0], [0, 0, 1, 0, 0, -1]]
    Generator system of M \cap N^n (number of generators 5):
    [[0, 0, 1, 0, 0, 1], [0, 1, 0, 0, 1, 0], [0, 3, 0, 4, 0, 0], 
    [1, 0, 0, 1, 0, 0], [4, 0, 0, 0, 3, 0]]
    Elasticity of S: 4/3
    Atoms of A(I_M) that reach the elasticity: 
        [[4, 0, 0, 0, 3, 0]]
    S has not acceptable elasticity (the monoid C is empty, 
    see step 6 of algorithm in Algorithm 28 in
    http://doi.org/10.1007/s00233-002-0022-4)
    False
\end{verbatim}
\end{example}


{\bf Acknowledgements}. The authors were supported partially by Junta de Andaluc\'{\i}a research groups FQM-343 and FQM-366, and by the project MTM2017-84890-P (MINECO/FEDER, UE).

\end{document}